\documentclass[oneside,english]{amsart}
\usepackage[T1]{fontenc}
\usepackage[latin9]{inputenc}
\usepackage{amstext}
\usepackage{amsthm}
\usepackage{amssymb}
\usepackage{lmodern,graphicx}

\makeatletter
\numberwithin{equation}{section}
\numberwithin{figure}{section}
\theoremstyle{plain}
\newtheorem{thm}{\protect\theoremname}
  \theoremstyle{plain}
  \newtheorem{prop}[thm]{\protect\propositionname}

\makeatother

\usepackage{babel}
  \providecommand{\propositionname}{Proposition}
\providecommand{\theoremname}{Theorem}

\begin{document}

\title[Unbounded and blow-up solutions for a delay logistic equation]{Unbounded and blow-up solutions for a delay logistic equation with positive feedback}

\author{Istv\'an Gy\H{o}ri} \address[Istv\'an Gy\H{o}ri]{University of Pannonia, Veszprem, Hungary}
\author{ Yukihiko Nakata} \address[Yukihiko Nakata]{Shimane University, Japan}
\author{ Gergely R\"ost} \address[Gergely R\"ost]{University of Szeged, Hungary and University
of Oxford, UK} \email{rost@math.u-szeged.hu}

\begin{abstract}
We study bounded, unbounded and blow-up solutions of a delay logistic equation without assuming the dominance of the instantaneous feedback. It is shown that there can exist an exponential (thus unbounded) solution for the nonlinear
problem, and in this case the positive equilibrium is always unstable. We obtain
a necessary and sufficient condition for the existence of blow-up solutions, and characterize a wide class of such solutions. There is a parameter set such
that the non-trivial equilibrium is locally stable but not globally stable due to
the co-existence with blow-up solutions. 
\end{abstract}

\maketitle

\section{Introduction}

There is a vast literature on the study of delay logistic equations describing
population growth of a single species \cite{Gopalsamy:1992,Kuang:1993}.
In \cite{Ruan:2006}, the qualitative studies of such logistic type delay differential
equations have been summarized, see also \cite{Appleby:2011,Faria:2003,Gyori:2000,Lenhart:1986,Teng:2002}
and references therein. 

In \cite{Lenhart:1986} the authors study the global asymptotic stability
of a logistic equation with multiple delays. Their global stability
result is generalized in Theorem 5.6 in Chapter 2 in \cite{Kuang:1993},
see also the discussion in \cite{Gyori:2000}. Those conditions presented
in \cite{Lenhart:1986} and in Theorem 5.6 in Chapter 2 in \cite{Kuang:1993}
are delay independent conditions, exploiting the dominance of the instantaneous
feedback. In \cite{Gyori:2000}, applying the oscillation theory of
delay differential equations \cite{GyoriLadas1991}, the first author
of this paper obtains a global stability condition for a logistic
equation without assuming the dominance of the instantaneous feedback.
See also \cite{GyoriLadas1991} and references therein for the study
of logistic equations without instantaneous feedback. 

When the instantaneous feedback term is small compared to the delayed
feedback, the positive equilibrium is not always globally stable.
Our motivation of this note comes from interesting examples shown
in \cite{Gyori:2000}, where the author shows the existence of an exponential
solution for a logistic equation with delay. Here we wish to investigate the properties of positive solutions of such an equation in detail. To be more specific, we
consider the logistic equation 
\begin{equation}
\frac{d}{dt}x(t)=r x(t)\left(1+\alpha x(t)-x(t-1)\right),\label{eq:Logistic}
\end{equation}
where $r>0$, $\alpha\in\mathbb{R}$. The equation (\ref{eq:Logistic})
is a special case of the equations studied in \cite{Gyori:2000}.
Note that (\ref{eq:Logistic}) is a normalized form of the following
delay differential equation 
\begin{equation}
\frac{d}{ds}N(s)=N(s)\left(\tilde  r+aN(s)-bN(s-\tau)\right),\label{eq:LogisticN}
\end{equation}
where $r>0,\ a\in\mathbb{R}$ and $b>0$. If we define $y(s):=\frac{b}{\tilde r}N(s)$ and $\alpha:=\frac{a}{b}$,
then we obtain
$$\frac{d}{ds} y(s)=\tilde r y(s) \left(1+\alpha y(s)-y(s-\tau)\right).$$
Next we scale the time so that the delay is normalized to be one, by letting
$s:=t \tau$ and $x(t):=y(s)$, then $y(s-\tau)=y(\tau(t-1))=x(t-1)$
and by calculating $\frac{d}{dt} x(t)$, 
we obtain
(\ref{eq:Logistic}) with $r=\tau \tilde r$.

For (\ref{eq:Logistic}) we show that there exist some unbounded solutions,
when $\alpha$ is allowed to be positive. More precisely, it is shown
that a blow-up solution (i.e. a solution that converges to infinity in finite time) exists if and only if $\alpha>0$ holds. We
then show that an exponential solution, namely $x(t)=ce^{rt},\ c>0$,
exists when $\alpha=e^{-r}$ holds, which is an unbounded but not blow-up solution.
The case is further elaborated, as we also find solutions which blows up faster
than the exponential solutions.

This paper is organized as follows. In Section 2 we collect previous
results on boundedness and stability, which are known in the literature, with the exception about the existence of blow-up solutions. In Section 3, we focus
on the exponential solution for the nonlinear differential equation
(\ref{eq:Logistic})  and its relation to stability. In Section 4 we characterize a large class of initial functions that generate superexponential blow-up solutions, and we also find a class of subexponential solutions. Section 5 is devoted to a summary and discussions.

\section{Boundedness, stability and blow-up}

Denote by $C$ the Banach space $C([-1,0],\mathbb{R})$ of continuous
functions mapping the interval $[-1,0]$ into $\mathbb{R}$ and designate
the norm of an element $\phi\in C$ by $\|\phi\|=\sup_{-1\leq\theta\leq0}|\phi(\theta)|$.
The initial condition for (\ref{eq:Logistic}) is a positive continuous
function given as 
\[
x(\theta)=\phi(\theta),\ \theta\in\left[-1,0\right].
\]

There exists a unique positive equilibrium given by
\[
x^{*}=\frac{1}{1-\alpha}
\]
if and only if $\alpha<1$ holds. In the following theorem we characterize
global and local dynamics of the solutions. The result on the existence of
a blow-up solution seems to be new. 
\begin{thm}
The following statements are true.

\begin{enumerate}
\item If 
\begin{equation}
\alpha\leq-1,\label{eq:Abs_cond}
\end{equation}
then the positive equilibrium is globally asymptotically stable.
\item If $-1<\alpha<1$, then the positive equilibrium is locally asymptotically
stable for 
\begin{equation}
r<\sqrt{\frac{1-\alpha}{1+\alpha}}\arccos(\alpha),\label{eq:delay_dep}
\end{equation}
and it is unstable for 
\begin{equation}
r>\sqrt{\frac{1-\alpha}{1+\alpha}}\arccos(\alpha).\label{eq:delay_dep2}
\end{equation}
Moreover, 

\begin{enumerate}
\item If $-1<\alpha\leq0$, then every solution is bounded.
\item If $0<\alpha$, then there exists a blow-up solution in a finite
time.
\end{enumerate}
\item If $\alpha\geq1$, then for every solution, which exists globally,
one has $\limsup_{t\to\infty}x(t)=\infty$.
\end{enumerate}
\end{thm}

\begin{proof}
1) For the global stability of the equilibrium we refer to the proof
of Theorem 5.6 in Chapter 2 in \cite{Kuang:1993}. 

2) The result for local asymptotic stability is well-known, see for
example Theorems 2 and 3 in \cite{Ruan:2006}.

2-a) Notice that a solution of \eqref{eq:Logistic} satisfies $x(t)=x(0)e^{r\int_0^t 1+\alpha x(s)-x(s-1) ds},$ hence positive solutions remain positive. When $-1<\alpha<0$, the result follows from a simple comparison
principle applied for the inequality $x'(t)\leq x(t)r(1-\alpha x(t))$.
In the case $\alpha=0$ we obtain the Wright's equation for which
boundedness is known, see \cite{Liz:2008}.

2-b) Let us assume that $\alpha >0$. We show that there exists
a solution that blows up in a finite time. Consider a positive continuous initial
function satisfying 
\[
\phi(\theta)=\begin{cases}
1, & \theta\in\left[-1,-\frac{1}{2}\right],\\
q, & \theta=0,
\end{cases}
\]
where $q$ is a positive constant to be determined later. Since one has
\[
1-x(t-1)=0,\ 0\leq t\le\frac{1}{2},
\]
it holds 
\begin{equation}
x^{\prime}(t)=r\alpha x(t)^{2},\ 0\leq t\leq\frac{1}{2}.\label{eq:ode_bu}
\end{equation}
Then equation (\ref{eq:ode_bu}) is easily integrated as 
\[
x(t)=\frac{1}{\frac{1}{q}-r\alpha t},
\]
for $t<\frac{1}{qr\alpha}$ where the solution exists. Let us set 
\[
q=\frac{h}{r\alpha},\ h\geq2,
\]
then we have 
\[
\lim_{t\uparrow\frac{1}{h}}x(t)=\lim_{t\uparrow\frac{1}{h}}\frac{1}{r\alpha(\frac{1}{h}-t)}=\infty,
\]
so the solution blows up at $t=\frac{1}{h}<\frac{1}{2}$.

3) The result follows from Theorem 5.1 in \cite{Gyori:2000}.

\end{proof}
\begin{figure}
\includegraphics[scale=0.85]{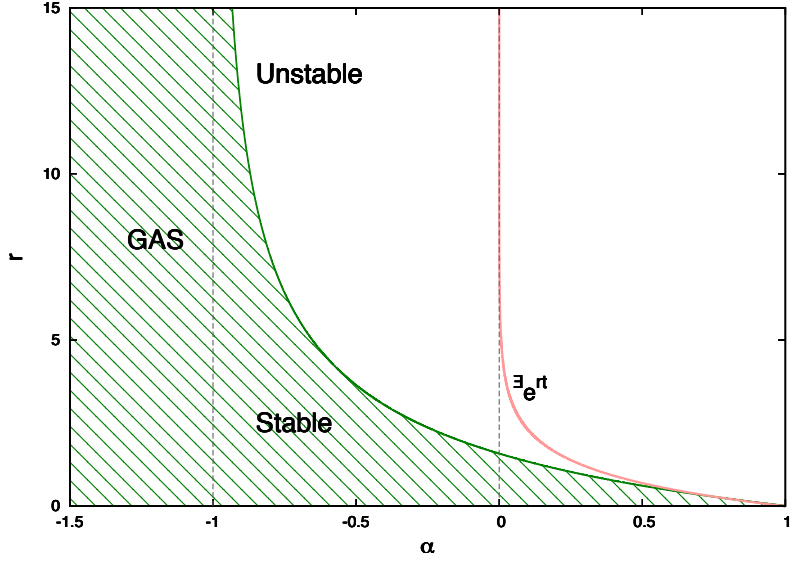}\caption{\label{fig:Stabr}Stability region for the positive equilibrium in
the $(\alpha,r)$-parameter plane. The shaded region is the stability
region given by (\ref{eq:Abs_cond}) and (\ref{eq:delay_dep}).
The positive equilibrium is globally stable for $\alpha\le-1$ and
is unstable above the stability boundary. Exponential solutions exist on the denoted curve. Blow-up solutions exist for $\alpha>0$, hence we can observe a region where the positive equilibrium is locally stable yet blow-up solutions also exist.}
\end{figure}

\section{Instability and exponential solutions}

A remarkable feature of the logistic equation with positive feedback
(\ref{eq:Logistic}) is the possible existence of exponential solutions, despite the equation being nonlinear. As in \cite{Gyori:2000}, we find the following.
\begin{prop}
\label{prop:exp}There exists an exponential solution 
\begin{equation}
x_{c}(t):=ce^{rt},\ t\geq-1\label{eq:exp}
\end{equation}
of the generalized logistic equation 
\begin{equation}
\frac{d}{dt}x(t)=r x(t)\left(1+\sum_{i=1}^{k} a_i x(t-\tau_i)\right),\label{eq:genlog}
\end{equation}
if and only if
\begin{equation}
\sum_{i=1}^{k} a_i e^{-r \tau_i}=0. \label{cond:genlog}
\end{equation}
\end{prop}
The proof is straightforward, and for \eqref{eq:Logistic} it means that an exponential solution exists if and only if
\begin{equation}
\alpha=e^{-r}\label{eq:ubd}
\end{equation}
holds.

It can be shown that the existence of the exponential solution implies
the instability of the positive equilibrium.
\begin{prop}
Let us assume that (\ref{eq:ubd}) holds. Then the positive equilibrium of (\ref{eq:Logistic})
is unstable.
\end{prop}

\begin{proof}
We compare the two conditions (\ref{eq:ubd}) and (\ref{eq:delay_dep2}).
We set 
\begin{equation}
\omega=\arccos(\alpha)\text{ for }\alpha\in\left(0,1\right).\label{eq:trans_om}
\end{equation}
Note that $\omega\in\left(0,\frac{\pi}{2}\right)$. Using the parameter
transformation (\ref{eq:trans_om}) we get 
\[
\arccos(\alpha)\sqrt{\frac{1-\alpha}{1+\alpha}}=\omega\frac{1-\cos\omega}{\sin\omega}
\]
and the condition (\ref{eq:ubd}) is written as $r=-\ln(\cos\omega)$.
Define
\begin{align*}
g_{1}(\omega) & :=\omega\frac{1-\cos\omega}{\sin\omega},\\
g_{2}(\omega) & :=-\ln(\cos\omega)
\end{align*}
for $\omega\in\left(0,\frac{\pi}{2}\right)$. We claim that 
\begin{equation}
g_{2}(\omega)>g_{1}(\omega),\ \text{\ensuremath{\omega\in\left(0,\frac{\pi}{2}\right)}}.\label{eq:cla}
\end{equation}
It is easy to see that $\lim_{\omega\downarrow0}g_{1}(\omega)=\lim_{\omega\downarrow0}g_{2}(\omega)=0$.
Straightforward calculations show
\begin{align*}
g_{1}^{\prime}(\omega) & =\frac{1-\cos\omega}{\sin\omega}\left(1+\frac{\omega}{\sin\omega}\right),\\
g_{2}^{\prime}(\omega) & =\frac{\sin\omega}{\cos\omega}.
\end{align*}
Then we see 
\begin{align*}
g_{2}^{\prime}(\omega)-g_{1}^{\prime}(\omega) & =\frac{1}{\cos\omega\sin\omega}\left\{ \sin^{2}\omega-\cos\omega\left(1-\cos\omega\right)\left(1+\frac{\omega}{\sin\omega}\right)\right\} \\
 & =\frac{1-\cos\omega}{\cos\omega\sin\omega}\left\{ \left(1+\cos\omega\right)-\cos\omega\left(1+\frac{\omega}{\sin\omega}\right)\right\} \\
 & =\frac{1-\cos\omega}{\cos\omega\sin\omega}\left(1-\omega\frac{\cos\omega}{\sin\omega}\right)\\
 & >0
\end{align*}
for $\omega\in\left(0,\frac{\pi}{2}\right)$. Thus we get (\ref{eq:cla})
and obtain the conclusion.
\end{proof}
In Figure \ref{fig:Stabr} we visualize the condition (\ref{eq:ubd})
in $(\alpha,r)$ parameter plane. Figure \ref{fig:Stabr} clearly
shows that if (\ref{eq:ubd}) holds, then the positive equilibrium
is unstable. 

\section{A new class of blow-up solutions}

We investigate other solutions of (\ref{eq:Logistic}) when (\ref{eq:ubd})
holds. 
\begin{thm}
\label{thm:blowup}Let the condition (\ref{eq:ubd}) hold. Consider
a solution with the initial function satisfying 
\begin{equation}
0<\phi(s)\leq ce^{rs},\ s\in\left[-1,0\right]\label{eq:IC1}
\end{equation}
with 
\begin{align}
\phi(0) & =c,\label{eq:IC2}\\
\phi(-1) & <ce^{-r}\label{eq:IC3}
\end{align}
for $c>0$. Then one has $x(t)>x_{c}(t)=ce^{rt}$ for $t>0$ and the
solution blows up at a finite time. 
\end{thm}

\begin{proof}
Looking for a contradiction, we assume that $x(t)$ exists on $\left[0,\infty\right)$.
Then $x(t)>0,\ t\geq-1$. Define 
\[
z(t):=\ln\left(\frac{x(t)}{x_{c}(t)}\right)=\ln\left(\frac{x(t)}{ce^{rt}}\right),\ t\geq-1.
\]
Let 
\[
\psi(s):=\ln\left(\frac{\phi(s)}{ce^{rs}}\right),\ s\in[-1,0].
\]
Then $z(s)=\psi(s)$ for $s\in\left[-1,0\right]$, and we have 
\[
0=\psi(0)\geq\psi(s), \quad 0=\psi(0)>\psi(-1),\ s\in\left[-1,0\right]
\]
from (\ref{eq:IC1}), (\ref{eq:IC2}) and (\ref{eq:IC3}). Now we
obtain the relation
$$\frac{d}{dt}z(t)=\frac{ce^{rt}}{x(t)}\frac{x'(t)ce^{rt}-x(t)rce^{rt}}{c^2e^{2rt}}=
\frac{x'(t)}{x(t)}-r ,$$
which can be rewritten by using \eqref{eq:Logistic} and $\alpha=e^{-r}$ as
$$\frac{d}{dt}z(t)=r (e^{-r} x(t)-x(t-1)),$$
and by $x(t)=e^{z(t)}ce^{rt}$ we obtain a nonautonomous differential equation for $z$: 
\begin{equation}
z^{\prime}(t)=rce^{r\left(t-1\right)}\left(e^{z(t)}-e^{z(t-1)}\right),\ t>0\label{eq:Non_logistic}
\end{equation}
using (\ref{eq:ubd}). First we show that $z^{\prime}(t)>0$ for any
$t\geq0$. Since 
\[
\psi(0)=\ln\left(\frac{\phi(0)}{c}\right)=0>\psi(-1)=\ln\left(\frac{\phi(-1)}{ce^{-r}}\right)
\]
follows from (\ref{eq:IC2}) and (\ref{eq:IC3}), one finds 
\[
z^{\prime}(0)=rce^{-r}\left(e^{\psi(0)}-e^{\psi(-1)}\right)>0.
\]
Assume that there exists $t_{1}>0$ such that $z^{\prime}(t)>0$ for $\ 0\le t<t_{1}$
and $z^{\prime}(t_{1})=0$ hold. If $t_{1}\in(0,1)$ then, since
$t_{1}-1\in\left(-1,0\right)$, 
\[
z^{\prime}(t_{1})=rce^{r\left(t_{1}-1\right)}\left(e^{z(t_{1})}-e^{z(t_{1}-1)}\right),
\]
while 
\[
z(t_{1})-z(t_{1}-1)=z(t_{1})-z(0)+\psi(0)-\psi(t_{1}-1)>0,
\]
thus we obtain a contradiction. If $t_{1}\geq1$, then 
\[
z(t_{1})-z(t_{1}-1)=\int_{t_{1}-1}^{t_{1}}z^{\prime}(s)ds>0,
\]
which leads a contradiction again. Therefore, we obtain $z^{\prime}(t)>0$
for $t\geq0$. 

We can fix a $T>2$ such that 
\[
1<\left(1-\alpha\right)re^{-r}ce^{rT}.
\]
Since $z(t)$ exists on $\left[0,\infty\right)$ and $z^{\prime}(t)>0$
for $t\geq0$, $z^{\prime}(t)>0$ for $0\leq t\leq T$. Thus
\[
m:=\min_{0\leq t\leq T}z'(t)>0.
\]
By the intermediate value theorem, for each $0\leq t\leq T$, there exists
$\xi(t)\in\left[z(t-1),z(t)\right]$ such that $e^{z(t)}-e^{z(t-1)}=e^{\xi(t)}\left(z(t)-z(t-1)\right)$.
Since $\xi(t)\geq z(t-1)>0$, we have 
\begin{align*}
z^{\prime}(t) & =rce^{r\left(t-1\right)}e^{\xi(t)}\left(z(t)-z(t-1)\right)\\
 & \geq rce^{r\left(t-1\right)}\left(z(t)-z(t-1)\right)\\
 & =rce^{r\left(t-1\right)}\int_{t-1}^{t}z^{\prime}(s)ds.
\end{align*}
This yields 
\begin{equation}
z'(t)>\int_{t-1}^{t}z^{\prime}(s)ds,\ t\geq T\label{eq:est_dif}
\end{equation}
and hence $z'(T)>\int_{T-1}^{T}z^{\prime}(s)ds\geq m$. This implies
that $z'(t)>m$ for $t\geq T$. Otherwise there is a $t_{1}>T>1$
such that $z^{\prime}(t)>m$ for $0<t<t_{1}$ and $z^{\prime}(t_{1})=m$.
But from (\ref{eq:est_dif})
\[
z'(t_{1})>\int_{t_{1}-1}^{t_{1}}z^{\prime}(s)ds>m,
\]
which is a contradiction.

Thus for any $t\geq1$ we have 
\begin{align*}
e^{z(t)}-e^{z(t-1)} & =e^{z(t)}\left(1-e^{-\left(z(t)-z(t-1)\right)}\right)\\
 & =e^{z(t)}\left(1-e^{-\int_{t-1}^{t}z^{\prime}(s)ds}\right)\\
 & \geq e^{z(t)}\left(1-e^{-m}\right).
\end{align*}
Therefore, 
\[
z^{\prime}(t)\geq rce^{r\left(t-1\right)}e^{z(t)}\left(1-e^{-m}\right),\ t\geq1,
\]
or equivalently 
\[
z^{\prime}(t)e^{-z(t)}\geq \left(rce^{-r}(1-e^{-m})\right)e^{rt},\ t\geq 1.
\]
Integrating both sides of
the above equation,
\begin{align*}
\int_{1}^{t}z^{\prime}(s)e^{-z(s)}ds & =\left[-e^{-z(s)}\right]_{s=1}^{s=t}=e^{-z(1)}-e^{-z(t)}\\
 & \geq  \left(rce^{-r}(1-e^{-m})\right) \int_{1}^{t}e^{rs}ds\\
 & = \left(ce^{-r}(1-e^{-m})\right) (e^{rt}-e^{r}),\ t>1.
\end{align*}
This yields 
\[
e^{-z(1)}\geq \left(ce^{-r}(1-e^{-m})\right) (e^{rt}-1)+e^{-z(t)}>\left(ce^{-r}(1-e^{-m})\right)(e^{rt}-1),\ t\geq1,
\]
which is a contradiction. Therefore, $z$ does not exist on $\left[0,\infty\right),$
moreover $z^{\prime}(t)\geq0,\ t\geq0$.

Consequently we should have a $T\in\left(0,\infty\right)$ such that
$\lim_{t\to T-}z(t)=+\infty$, that is it is a blow-up solution. Corresponding
to this blow-up solution, we also have 
\[
x(t)=ce^{z(t)+rt}\to\infty,\ t\to T^{-},
\]
that is $x(t)$ is also a blow-up solution. 
\end{proof}

\begin{prop}
The following estimate is valid: \[
\frac{1}{r}\ln\left(1+\frac{e^{r}}{c}\right)\leq T,
\]
where $T$ is the blow-up time for a blow-up solution $x$  in Theorem
\ref{thm:blowup}.

\end{prop}

\begin{proof}
We can find the lower bound for the blow-up time $T$ by the standard comparison principle. 
Consider the following ordinary differential equation 
\[
y^{\prime}(t)=y(t)r\left(1+e^{-r}y(t)\right)
\]
with $y(0)=c=x(0)$. By the comparison theorem, we have $x(t)\leq y(t),\ t\geq0$.
Integrating the equation, we get 
\[
y(t)=\frac{ce^{rt}}{1+\left(1-e^{rt}\right)e^{-r}c}
\]
for sufficiently small $t$. From this expression, we find the finite
blow-up time for $y$ and then we obtain the required estimation.
\end{proof}

Similar to the proof of Theorem \ref{thm:blowup}, we obtain the following
theorem. 
\begin{thm}
\label{thm:global}Let the condition (\ref{eq:ubd}) hold. Consider
a solution with the initial function satisfying 
\begin{equation}
ce^{rs}\leq\phi(s),\ s\in\left[-1,0\right]\label{eq:IC1-1}
\end{equation}
with 
\begin{align}
\phi(0) & =c,\label{eq:IC2-1}\\
\phi(-1) & >ce^{-r}\label{eq:IC3-1}
\end{align}
for $c>0$. Then one has $x(t)<x_{c}(t)=ce^{rt}$ for $t>0$, and consequently the solution exists on $[-1,\infty)$.
\end{thm}

For the initial functions considered in Theorems \ref{thm:blowup}
and \ref{thm:global}, $\frac{x(t)}{x_{c}(t)}$ is a monotone function
for $t>0$, thus the order of the solution with respect to the exponential
solution, $x_{c}(t)=ce^{rt}$, is preserved. We do not analyze the
qualitative behavior of the solution with the initial condition that
oscillates about the exponential solution. Numerical simulations suggest
that, for many solutions, $\frac{x(t)}{x_{c}(t)}$ eventually becomes
a monotone function. 

\section{Discussion}

In this paper we study the logistic equation (\ref{eq:Logistic}).
In the stability analysis of delayed logistic equations, negative
instantaneous feedback is usually assumed, see \cite{Faria:2003,Gyori:2000,Ruan:2006,Teng:2002}
and references therein. Only a few stability results are available in the
literature for the case of positive instantaneous feedback e.g., \cite{Lenhart:1986,Li:2010}.
However, the blow-up solutions, which are present due to the positive instantaneous
feedback, have not been analyzed in detail, since the publication of
the paper \cite{Gyori:2000}. This manuscript has been inspired by the work
done in \cite{Gyori:2000}, especially, paying attention to the examples and open questions
given in Section 5 of the paper \cite{Gyori:2000}. Our primary goal
was to clarify and understand a relation between the stability condition
of the positive equilibrium and the existence condition for the exponential
solution. For the logistic equation (\ref{eq:Logistic}), we show
that the existence of the exponential solution implies instability
of the positive equilibrium in Proposition \ref{prop:exp}, see also
Fig. \ref{fig:Stabr}. Since stability analysis becomes extremely
hard for the differential equation with multiple delays, the comparison
of the existence condition of the exponential solution to the stability
condition is not straightforward in general, thus it remains an
open problem whether the positive equilibrium of \eqref{eq:genlog} is always unstable whenever exists and \eqref{cond:genlog} holds. 

Finding a global stability condition for  \eqref{eq:Logistic} in the case of $-1<\alpha<0$ is still an open problem.
For $\alpha=0$ the global stability problem is known as the famous
Wright conjecture \cite{krisztin}. On the other hand, for $0<\alpha<1$, due to the
existence of the blow-up solution, it is shown that the stable equilibrium
can not attract every solution, thus there is no hope to obtain global
stability condition for $0<\alpha<1$. Numerical simulations also
suggest that there are many bounded and oscillatory solutions 
. 

In Theorems \ref{thm:blowup} and \ref{thm:global} we fix the parameters
as in (\ref{eq:ubd}) in Proposition \ref{prop:exp} so that the exponential
solution exists for the logistic equation (\ref{eq:Logistic}). We
consider some solutions that preserve the order with respect to the
exponential solution, and show that some blow up, while others exist for all positive time. The qualitative behaviour of the solution with
the initial condition that oscillates about the exponential solution
is not studied. For such an initial function, careful estimation of
the solution seems to be necessary to understand the long term solution behaviour.
Numerical simulations suggest that, for many solutions, $\frac{x(t)}{x_{c}(t)}$
eventually becomes a monotone function. The detailed understanding of the evolution of such solutions is also left for future work. 

Finally, one might be interested in the equation 
\[
\frac{d}{dt}y(t)=y(t)r\left(1+\alpha y(t)+y(t-1)\right),
\]
which has an opposite sign for the delayed feedback term. For this
equation the qualitative dynamics is studied in the literature. The
positive equilibrium
\[
y^{*}=-\frac{1}{1+\alpha}
\]
exists if and only if $\alpha<-1$. According to Theorem 5.6 in Chapter
2 in \cite{Kuang:1993}, the positive equilibrium is globally asymptotically
stable. When $\alpha\geq-1$ the solutions are unbounded, see again
Theorem 5.1 in \cite{Gyori:2000}. 

\subsection*{Acknowledgement}
The first author\textquoteright s research has been
supported by the Hungarian National Research Fund Grant OTKA K120186 and the  
Sz\'echenyi 2020 project EFOP-3.6.1-16-2016-00015.
The second author was supported by JSPS Grant-in-Aid for Young Scientists
(B) 16K20976.  The third author was supported by Hungarian National Research Fund Grant NKFI FK 124016 and MSCA-IF 748193. The meeting of the authors have been supported by JSPS and NKFI Hungary-Japan bilateral cooperation project.

\end{document}